\documentclass[11pt,letterpaper]{article}
\usepackage{amsthm}
\usepackage{amsmath}
\usepackage{amsfonts}
\usepackage{latexsym}
\usepackage{geometry}
\usepackage{enumerate}
\usepackage{csquotes}

\emergencystretch=\maxdimen
\title{The rigidity of $\mathbb{S}^3\times\mathbb{R}$ under ancient Ricci flow}
\author{Yongjia Zhang}
\numberwithin{equation}{section}
\begin{document}
\maketitle

In this paper we generalize the neck-stability theorem of Kleiner-Lott \cite{kleiner2014singular} to a special class of four-dimensional nonnegatively curved Type I $\kappa$-solutions, namely, those whose asymptotic shrinkers are the standard cylinder $\mathbb{S}^3\times\mathbb{R}$. We use this stability result to prove a rigidity theorem: if a four-dimensional Type I $\kappa$-solution with nonnegative curvature operator has the standard cylinder $\mathbb{S}^3\times\mathbb{R}$ as its asymptotic shrinker, then it is exactly the cylinder with its standard shrinking metric.

\tableofcontents

\section{Introduction}
The Ricci flow has been playing an important role in geometric analysis ever since its invention by Hamilton \cite{hamilton1982three}. Perelman's insightful understanding of the Ricci flow in dimension three leads to his solution to the geometrization conjecture \cite{perelman2002entropy} \cite{perelman2003ricci}. Today the Ricci flow in higher dimensions, especially in dimension four, is still of great interest and importance, and yet not well-understood. Never too much can be done to understand the ancient solutions and to study the singularity formation of the Ricci flow in higher dimensions. As a continuation of classification of three-dimensional Type I $\kappa$-solutions \cite{ni2009closed} \cite{zhang2017three}, we work on four-dimensional Type I $\kappa$-solutions in this paper. First recall the following definition.

\newtheorem{Definition_1}{Definition}[section]
\begin{Definition_1} \label{Definition_1}
An ancient solution to the Ricci flow $(M,g(t))_{t\in(-\infty,0]}$ is called a \emph{$\kappa$-solution} if it is $\kappa$-noncollapsed on all scales and has bounded curvature on every time slice. A $\kappa$-solution is called \emph{Type I} if its Riemann curvature tensor satisfies
\begin{eqnarray} \label{eq:Type_I}
|Rm\left(g(t)\right)|\leq\frac{C_I}{|t|},
\end{eqnarray}
for all $t\in(-\infty,0)$, where $C_I<\infty$ is a constant that we name as the \emph{curvature bound coefficient}.
\end{Definition_1}

\bigskip

It is known that every $\kappa$-solution either with nonnegative curvature operator or with Type I curvature bound has a backward blow-down shrinking gradient Ricci soliton, or \textit{asymptotic shrinker} for short; see Perelman \cite{perelman2002entropy} and Naber \cite{naber2010noncompact}. We deal with only the case when the asymptotic shrinker is the cylinder $\mathbb{S}^3\times\mathbb{R}$. Our main theorem is the following.

\newtheorem{Theorem_Main}[Definition_1]{Theorem}
\begin{Theorem_Main}\label{Theorem_Main}
Let $(M^4,g(t))_{t\in(-\infty,0]}$ be a four-dimensional simply connected nonnompact Type I $\kappa$-solution to the Ricci flow with nonnegative curvature operator. If an asymptotic shrinker of $(M,g(t))$ is $\mathbb{S}^3\times\mathbb{R}$ with its standard metric and standard potential function, then $(M,g(t))$ is the same cylinder with the standard shrinking metric.
\end{Theorem_Main}

\bigskip

\noindent\textbf{Remarks:}
\begin{enumerate}[(1)]
\item In the case of dimension three we were able to deal with any Type I $\kappa$-solution, not only because we automatically have nonnegative curvature operator by the Hamilton-Ivey pinching estimate, but also because there are only two noncompact Ricci shrinkers, $\mathbb{S}^2\times\mathbb{R}$ and its $\mathbb{Z}_2$ quotient, and the latter does not arise as asymptotic shrinker unless from itself. The assumption on the asymptotic shrinker made in Theorem \ref{Theorem_Main} gives a lower bound of the reduced volume function, which rules out all other possibility of blow-down limits; see section 3 below for more details.

\item The condition of Type I curvature bound cannot be dropped, since the Bryant soliton is an obvious counterexample. The condition of nonnegative curvature operator can be dropped in the a priori estimates of the reduced distance function but is needed for some geometric purpose, for instance splitting at infinity \cite{perelman2002entropy}, the Gromoll-Meyer theorem, Hamilton's maximum principle \cite{hamilton1986four}, and Perelman's $\kappa$-compactness theorem \cite{perelman2002entropy}.
\end{enumerate}
\bigskip

Before we proceed to introduce our principal technique, we introduce the definition of an $\varepsilon$-neck.

\newtheorem{Definition_2}[Definition_1]{Definition}
\begin{Definition_2} \label{Definition_2}
Let $(M^4,g(t))$ be a solution to the Ricci flow. A point $(x,t)$ is called the center of a (strong) $\varepsilon$-neck, if on the space-time product $\displaystyle B_{g(t)}(x,\varepsilon^{-1}R(x,t)^{-\frac{1}{2}})\times[t-R(x,t)^{-1},t]$, the Ricci flow $g(t)$ is defined, and, after scaling by $R(x,t)$ and shifting time, is $\varepsilon$-close to a correspondent piece of a shrinking round cylinder $(\mathbb{S}^3\times\mathbb{R},g_{cyl}(t),(x_0,-1))_{t\in(-\infty,0)}$ in the $\displaystyle C^{\lfloor\varepsilon^{-1}\rfloor}$ topology. Here (and in the rest of this paper) $R$ stands for the scalar curvature.
\end{Definition_2}

\bigskip

The following generalized neck-stability theorem, modelled on Kleiner-Lott \cite{kleiner2014singular}, is our main technique in the proof of the main theorem.
\\

\newtheorem{Proposition_Main}[Definition_1]{Proposition}
\begin{Proposition_Main}\label{Proposition_Main}
There is a positive function $\delta=\delta(\kappa,C_I)$ with the following property. For any $\kappa>0$, any $C_I<\infty$, and any small positive numbers $\delta_0$, $\delta_1\in (0,\delta(\kappa,C_I))$, there exists $T=T(\kappa,\delta_0,\delta_1,C_I)\in(-\infty,0)$, such that the following holds. suppose
\begin{enumerate}[(a)]
  \item $(M,g(t))_{t\in(-\infty,0]}$ is a noncompact Type I $\kappa$-solution to the Ricci flow with nonnegative curvature operator, and
  \item An asymptotic shrinker of $(M,g(t))$ is the standard cylinder $\mathbb{S}^3\times\mathbb{R}$, and
  \item The curvature bound coefficient of $(M,g(t))$ is less than or equal to $C_I$, and
  \item $x_0\in M$, $R(x_0,0)=1$, and $x_0$ is the center of a $\delta_0$-neck.
\end{enumerate}
Then for all $t<T$, $(x_0,t)$ is the center of a $\delta_1$-neck.
\end{Proposition_Main}

\bigskip

\noindent\textbf{Remark:} Proposition \ref{Proposition_Main} holds if one replaces the nonnegative curvature operator condition with the condition that $M\times\mathbb{R}^2$ has nonnegative isotropic curvature, which was considered by Brendle \cite{brendle2009generalization}.

\bigskip

This paper is organized as follows. In section 2 we introduce some fundamental results on the $\mathcal{L}$-geometry. In section 3 we prove Proposition \ref{Proposition_Main}. In section 4 we apply Proposition \ref{Proposition_Main} to prove Theorem \ref{Theorem_Main}. In this paper, we always work with noncompact $\kappa$-solutions.
\\

\section{Preliminaries on the $\mathcal{L}$-geometry}

In this section we review some basic facts of Perelman's reduced geometry \cite{perelman2002entropy}, of which \cite{chow2007ricci} and \cite{morgan2007ricci} are good references. Let $(M,g(t))_{t\in[T,0]}$ be a Ricci flow and $x_0$ be a fixed point on $M$. Then the $L$ function based at $(x_0,0)$ is defined as
\begin{eqnarray*}
L_{(x_0,0)}(x,t)=\inf_{\gamma}\int_{0}^{|t|}\sqrt{\tau}\left(R(\gamma(\tau),-\tau)+|\dot{\gamma}(\tau)|_{g(-\tau)}^2\right)d\tau,
\end{eqnarray*}
where $t\in[T,0)$, $x\in M$, and the infimum is taken among all smooth curves $\gamma:[0,|t|]\rightarrow M$ with $\gamma(0)=x_0$ and $\gamma(|t|)=x$. The minimizing curve is called a shortest $\mathcal{L}$-geodesic, which satisfies a second order linear ordinary differential equation depending only on the local geometry. Then the \textit{reduced distance} function based at $(x_0,0)$ is defined as
\begin{eqnarray*}
l_{(x_0,0)}(x,t)=\frac{1}{2\sqrt{|t|}}l_{(x_0,0)}(x,t).
\end{eqnarray*}
For notational simplicity, we often write $l_{(x_0,0)}$ as $l$ when the base point is understood. From Perelman \cite{perelman2002entropy} we know that the $l$ function satisfies some nice differential inequalities, among which are
\begin{eqnarray}
-\frac{\partial}{\partial t}l-\Delta l+|\nabla l|^2-R-\frac{n}{2t}&\geq& 0,\label{eq:diffe_eq_1}
\\
2\Delta l- |\nabla l|^2+R-\frac{l-n}{t}&\leq& 0.
\end{eqnarray}
The most striking fact is the monotonicity of the following \textit{reduced volume}
\begin{eqnarray*}
\mathcal{V}_{(x_0,0)}(g(t))=\frac{1}{(4\pi|t|)^{\frac{n}{2}}}\int_M e^{-l_{(x_0,0)}(x,t)}dg_t(x).
\end{eqnarray*}
$\mathcal{V}_{(x_0,0)}(g(t))$ is always monotonically nondecreasing in $t$ under the Ricci flow. We often write $\mathcal{V}_{(x_0,0)}(g(t))$ as $\mathcal{V}(t)$ for simplicity.
\\

It is known that for ancient solutions to the Ricci flow, either under the assumption of nonnegative curvature operator or under the assumption of Type I curvature bound, the reduced distance can be estimated. The most important estimates are the following:
\begin{eqnarray}
|\nabla l|^2+R&\leq&\frac{C(l+1)}{|t|}, \label{eq:l_1}
\\
-C(l(x,t)+1)+\frac{c}{|t|}dist_{g(t)}(x,y)^2\leq l(y,t)&\leq&C(l(x,t)+1)+\frac{C}{|t|}dist_{g(t)}(x,y)^2, \label{eq:l_2}
\end{eqnarray}
where $l$ is the reduced distance function based at some point at $t=0$. In the case of nonnegative curvature operator, the constants $c$ and $C$ depend only on the dimension, whereas they also depend on $C_I$ in the Type I case. One may refer to Morgan-Tian \cite{morgan2007ricci} and Naber \cite{naber2010noncompact} for (\ref{eq:l_1}) and (\ref{eq:l_2}) in these two cases, respectively. Among these estimates, the least obvious one is the first inequality in (\ref{eq:l_2}); one may in particular refer to Lemma 9.25 in \cite{morgan2007ricci} for its proof. We emphasize here that by (\ref{eq:l_2}), centered at points where $l$ is uniformly bounded, the integrand $(4\pi|t|)^{-\frac{n}{2}}e^{-l}$ of the reduced volume behaves nicely as a Gaussian kernel, hence, as we will show below, along these points, the reduced volumes also converge along with the pointed smooth convergence of the Ricci flows.
\\

We collect some basic results concerning the reduced geometry and compactness of $\kappa$-solutions. They are modified from some well-known results from the literature; see \cite{kleiner2008notes}, \cite{kleiner2014singular}, and \cite{morgan2007ricci} for instance.
\\

\newtheorem{Lemma_2_1}{Lemma}[section]
\begin{Lemma_2_1} \label{Lemma_2_1}
Let $\displaystyle\{(M_i^n,g_i(t),(x_i,0))_{t\in(-\infty,0]}\}_{i=1}^\infty$ be a sequence of $\kappa$-solutions of dimension $n\geq 3$ with nonnegative curvature operator. Let $l_i$ be the reduced distance based at $(x_i,0)$ and $\mathcal{V}_i$ be the correspondent reduced volume. Assume that $\limsup_{i\rightarrow\infty} R(x_i,0)<\infty$ and that every $g_i(t)$ is Type I with respect to a universal curvature bound coefficient $C_I$. Then by passing to a subsequence we have that
\begin{eqnarray*}
(M_i,g_i(t),(x_i,0),l_i)\rightarrow (M_\infty,g_\infty(t),(x_\infty,0),l_\infty).
\end{eqnarray*}
Here $(M_\infty,g_\infty(t),(x_\infty,0))_{t\in(-\infty,0]}$ is a $\kappa$-solution of Type I with $C_I$ being the curvature bound coefficient, and $l_\infty$ is the reduced distance function based at $(x_\infty,0)$ satisfying (\ref{eq:diffe_eq_1})---(\ref{eq:l_2}), which are understood in the sense of distribution whenever necessary. The convergence of the Ricci flows is in the pointed smooth Cheeger-Gromov-Hamilton sense \cite{hamilton1995compactness}. The convergence of $\{l_i\}_{i=1}^\infty$ is in the $C_{loc}^\alpha$ or weak $\left(W_{loc}^{1,2}\right)^*$ sense on $M_\infty\times(-\infty,0)$.  Moreover, $\mathcal{V}_i(t)$ converges locally uniformly on $(-\infty,0)$ to $\mathcal{V}_\infty(t)$, where
\begin{eqnarray*}
\mathcal{V}_\infty(t)=\frac{1}{(4\pi|t|)^{\frac{n}{2}}}\int_{M_\infty}e^{-l_\infty(x,t)}dg_\infty(t).
\end{eqnarray*}
\end{Lemma_2_1}

\bigskip

\noindent\textbf{Remarks:}
\begin{enumerate}[(1)]
  \item When the dimension is higher than $3$, Perelman \cite{perelman2002entropy} only proved a precompactness theorem for $\kappa$-solutions, to wit, one does not know whether or not the curvature is bounded on the limit flow. Hence if one drops the uniformly Type I condition for this sequence, one still obtains a similar compactness theorem, except that the limit may not have bounded curvature.
  \item In this lemma, $(M_\infty,g_\infty(t),(x_\infty,0))_{t\in(-\infty,0]}$ may not have bounded curvature at time $0$ slice. If this is the case, we point out that since $l_\infty$ depends only on local geometry, it is still well-defined and satisfies (\ref{eq:diffe_eq_1})---(\ref{eq:l_2}).
\end{enumerate}

\begin{proof}[Proof of Lemma \ref{Lemma_2_1}]
The subconvergence of the ancient flows $(M_i,g_i(t),(x_i,0))\rightarrow (M_\infty,g_\infty(t),(x_\infty,0))$ is a standard result by Perelman \cite{perelman2002entropy}, and the Type I condition is carried to the limit by the convergence. To prove the convergence of the reduced distance and the reduced volume, we fix an arbitrary $A\in(1,\infty)$, and consider any $t\in[-A,-A^{-1}]$. By Proposition 2.1 of Naber \cite{naber2010noncompact}, we have that
\begin{eqnarray}\label{eq:estimates}
l_i(x_i,t)&\leq& C,
\end{eqnarray}
where $C$ is a  positive constant depending only on $C_I$. From the estimates (\ref{eq:l_1}) and (\ref{eq:l_2}), the convergence of $l_i$ follows immediately. Moreover, since $l_i$ is determined by the $\mathcal{L}$-geodesics, which depend only on local geometry, from the convergence of the Ricci flows we know that $l_\infty$ is the reduced distance based at $(x_\infty,0)$. By (\ref{eq:l_2}), the integrand $(4\pi|t|)^{-\frac{n}{2}}e^{-l_i(x,t)}$ of the reduced volume has uniform rapid decay at infinity. Because of the curvature nonnegativity, there is a uniform volume growth upper bound for every $g_i(t)$ by the Bishop-Gromov volume comparison theorem. It follows that for every $\eta>0$, there exists a $B_0<\infty$ depending also on $A$, such that
\begin{eqnarray*}
\mathcal{V}_i(t)-\eta\leq \int_{B_{g_i(t)}(x_i,B_0)}(4\pi|t|)^{-\frac{n}{2}}e^{-l_i(x,t)}dg_i(t)&\leq&\mathcal{V}_i(t),
\\
\mathcal{V}_\infty(t)-\eta\leq \int_{B_{g_\infty(t)}(x_\infty,B_0)}(4\pi|t|)^{-\frac{n}{2}}e^{-l_\infty(x,t)}dg_\infty(t)&\leq&\mathcal{V}_\infty(t),
\\
(4\pi|t|)^{-\frac{n}{2}}\left|\int_{B_{g_i(t)}(x_i,B_0)}e^{-l_i(x,t)}dg_i(t)-\int_{B_{g_\infty(t)}(x_\infty,B_0)}e^{-l_\infty(x,t)}dg_\infty(t)\right|&\leq&\eta,
\end{eqnarray*}
for all $t\in[-A,-A^{-1}]$ and for all $i$ large enough. The convergence of the reduced volume follows immediately.
\end{proof}

\bigskip

The following lemma differs from Lemma \ref{Lemma_2_1} in the choice of the base points. The convergence of ancient flows follows directly from the uniform Type I condition instead of Perelman's compactness theorem.

\newtheorem{Lemma_2_2}[Lemma_2_1]{Lemma}
\begin{Lemma_2_2}\label{Lemma_2_2}
Let $\{(M_i^n,g_i(t),(x_i,0))_{t\in(-\infty,0]}\}_{i=1}^\infty$ be a sequence of $\kappa$-solutions of dimension $n\geq3$ with nonnegative curvature operator. Let $l_i$ be the reduced distance based at $(x_i,0)$ and $\mathcal{V}_i$ be the correspondent reduced volume. Assume that $\limsup_{i\rightarrow\infty}l_i(y_i,-1)<\infty$, where $y_i\in M_i$, and that every $g_i(t)$ is Type I with respect to a universal curvature bound coefficient $C_I$. Then by passing to a subsequence, we have that
\begin{eqnarray*}
(M_i,g_i(t),(y_i,-1),l_i)\rightarrow (M_\infty,g_\infty(t),(y_\infty,-1),l_\infty),
\end{eqnarray*}
in the same sense as in Lemma \ref{Lemma_2_1}. Moreover, the following hold.
\begin{enumerate}[(a)]
  \item $(M_\infty,g_\infty(t),(y_\infty,-1))_{t\in(-\infty,0)}$ is either a nonnegatively curved Type I $\kappa$-solution with the curvature bound coefficient $C_I$, or the Gaussian shrinker, that is, the static Euclidean space.
  \item $l_\infty$ is a locally Lipschitz function on $M_\infty\times(-\infty,0)$ satisfying (\ref{eq:diffe_eq_1})---(\ref{eq:l_2}), which are understood in the sense of distribution whenever necessary.
  \item $\mathcal{V}_i(t)$ converges locally uniformly on $(-\infty,0)$ to $\mathcal{V}_\infty(t)$, where
      \begin{eqnarray*}
      \mathcal{V}_\infty(t)=\frac{1}{(4\pi|t|)^{\frac{n}{2}}}\int_{M_\infty}e^{-l_\infty(x,t)}dg_\infty(t)
      \end{eqnarray*}
  \item If $\mathcal{V}_\infty$ is a constant on an interval $[t_0,t_1]\subset(-\infty,0)$, then $(M_\infty,g_\infty(t),l_\infty)$ is the canonical form of a shrinking gradient Ricci soliton with the potential function being $l_\infty$.
\end{enumerate}
\end{Lemma_2_2}

\noindent\textbf{Remarks:}
\begin{enumerate}[(1)]
  \item In the above lemma either the nonnegative curvature operator assumption or the uniform Type I assumption may be dropped; there is a similar estimate as (\ref{eq:l_3}) below in the Type I case; see Lemma 2.5 in Naber \cite{naber2010noncompact}. Notice that an individual Type I assumption is not sufficient for this result since the estimates in Naber \cite{naber2010noncompact} depend on $C_I$.
  \item If $\mathcal{V}_\infty$ is indeed a constant for a period of time, then by Perelman's monotonicity formula \cite{perelman2002entropy} one may obtain that
      \begin{eqnarray*}
      2\Delta l_\infty -|\nabla l_\infty|^2+R+\frac{l_\infty-n}{-t}=0,
      \end{eqnarray*}
      see Proposition 9.20 in Morgan-Tian \cite{morgan2007ricci}. It follows that the potential function $l_\infty$ is normalized in the way that $\mathcal{V}_\infty$ is equal to the asymptotic reduced volume in the sense of Yokota \cite{yokota2009perelman} \cite{yokota2012addendum} or Gaussian density in the sense of Cao-Hamilton-Ilmanen \cite{cao2004gaussian}.
\end{enumerate}

\begin{proof}[Proof of Lemma \ref{Lemma_2_2}]
By Corollary 9.10 in Morgan-Tian \cite{morgan2007ricci}, we have that
\begin{eqnarray}\label{eq:l_3}
\left|\frac{\partial}{\partial t}l\right|\leq \frac{C}{|t|}l,
\end{eqnarray}
where $C$ is a constant. By integrating (\ref{eq:l_3}) we have that for all $A\in(0,\infty)$, there exists $C(A)<\infty$, such that
\begin{eqnarray*}
l_i(y_i,t)\leq C(A),
\end{eqnarray*}
whenever $t\in[-A,-A^{-1}]$ and for all $i$. The convergence of $l_i$ and $\mathcal{V}_i$ follows from (\ref{eq:l_1}) and (\ref{eq:l_2}) in the same manner as in the proof of Lemma \ref{Lemma_2_1}. (d) is a simple consequence of Perelman's monotonicity formula.
\end{proof}

\bigskip

The following lemma indicates the behaviour of the reduced distance on the shrinking round cylinder. This Lemma is the same as Lemma 6.7 in Kleiner-Lott \cite{kleiner2014singular}. Even though they worked with three-dimensional cylinder, since there is no essential geometric difference, the same holds on the four-dimensional cylinder.

\newtheorem{Lemma_2_3}[Lemma_2_1]{Lemma}
\begin{Lemma_2_3} \label{Lemma_2_3}
Let $(M,g(t),(x,0))_{t\in(-\infty,0]}$ be the shrinking round cylinder $\mathbb{S}^3\times\mathbb{R}$ with $R\equiv 1$ at $t=0$. Let $l$ be the reduced distance based at $(x,0)$. Let $\tau_i\nearrow\infty$ be a sequence of positive numbers. Then we have
\begin{eqnarray*}
(M,\tau_i^{-1}g(t\tau_i),(x,-1),l(t\tau_i))\rightarrow (Cyl,g_{cyl}(t),(x_0,-1),l_{cyl})
\end{eqnarray*}
in the same sense as in Lemma \ref{Lemma_2_1}. Moreover, the following hold.
\begin{enumerate}[(a)]
  \item $\displaystyle \lim_{t\rightarrow\infty} l(x,t)=\frac{3}{2}$;
  \item $(Cyl,g_{cyl}(t))_{t\in(-\infty,0)}$ is also the shrinking round cylinder $\mathbb{S}^3\times\mathbb{R}$ that forms a forward singularity at $t=0$;
  \item If we denote the coordinate in the $\mathbb{R}$ direction on $\mathbb{S}^3\times\mathbb{R}$ as $z$, such that $z(x_0)=0$, then we have
      \begin{eqnarray} \label{eq:l_cyl}
      l_{cyl}(y,t)=\frac{3}{2}+\frac{z^2}{-4t}.
      \end{eqnarray}
\end{enumerate}
\end{Lemma_2_3}

\bigskip

From Lemma \ref{Lemma_2_2}(d) we know that $(Cyl,g_{cyl}(t),(x_0,-1),l_{cyl})$ is (the normalized canonical form of) the cylindrical shrinker $\mathbb{S}^3\times\mathbb{R}$; we will use this notation in the rest of the paper.

\section{A generalized neck-stability theorem}

In this section we prove Proposition \ref{Proposition_Main}. Our argument is modelled on section 6 of Kleiner-Lott \cite{kleiner2014singular}. The first question one might ask is whether or not the asymptotic shrinker is unique. The following lemma shows that this is the case at least when the asymptotic shrinker is the round cylinder.
\\

\newtheorem{Lemma_3}{Lemma}[section]
\begin{Lemma_3}\label{Lemma_3}
Let $(M^4,g(t))_{t\in(-\infty,0]}$ be a $\kappa$-solution to the Ricci flow with nonnegative curvature operator. Let $l$ and $\bar{l}$ be reduced distances based at points in $M\times\{0\}$ and $M\times(-\infty,0]$, respectively. Let $\{p_i\}_{i=1}^\infty\subset M$, $\{\bar{p}_i\}_{i=1}^\infty\subset M$, $\tau_i\nearrow\infty$, and $\bar{\tau}_i\nearrow\infty$ be such that $\limsup_{i\rightarrow\infty}l(p_i,-\tau_i)<\infty$ and $\limsup_{i\rightarrow\infty}\bar{l}(\bar{p}_i,-\bar{\tau}_i)<\infty$. Then if
\begin{eqnarray*}
(M,\tau_i^{-1}g(t\tau_i),(p_i,-1),l(t\tau_i))\rightarrow (Cyl,g_{cyl}(t),(y,-1),l_{cyl}),
\end{eqnarray*}
for some $y\in Cyl$ along some subsequence. Then
\begin{eqnarray*}
(M,\bar{\tau}_i^{-1}g(t\bar{\tau}_i),(\bar{p}_i,-1),\bar{l}(t\bar{\tau}_i))\rightarrow (Cyl,g_{cyl}(t),(y',-1),l_{cyl}),
\end{eqnarray*}
along any convergent subsequence, where $y'\in Cyl$ may be different from $y$ and may be dependent on the subsequence.
\end{Lemma_3}

\begin{proof}
We denote the reduced distance correspondent to $l$ and $\bar{l}$ as $\mathcal{V}$ and $\bar{\mathcal{V}}$, respectively. By Lemma 3.1 in Yokota \cite{yokota2009perelman}, we have that $\lim_{t\rightarrow-\infty}\mathcal{V}(t)\leq\lim_{t\rightarrow-\infty}\bar{\mathcal{V}}(t)$. It follows from this fact and Lemma \ref{Lemma_2_2} that the limit of $\{(M,\bar{\tau}_i^{-1}g_i(t\bar{\tau}_i),(\bar{p}_i,-1))_{t\in(-\infty,0)}\}_{i=1}^\infty$ is the canonical form of a noncompact shrinking gradient Ricci soliton with nonnegative curvature operator, whose asymptotic reduced volume (Gaussian density) is no less than that of the shrinking cylinder. By the classification of four-dimensional shrinkers with nonnegative curvature operator (see \cite{naber2010noncompact}, \cite{li2016four}, and \cite{munteanu2017positively}), and by the value table given in \cite{cao2004gaussian}, we know this limit must be the shrinking cylinder.
\end{proof}

\bigskip

In this section, we consider all the four-dimensional nonnegatively curved Type I $\kappa$-solutions with respect to a universal curvature bound coefficient, and with their asymptotic shrinkers being the cylinder $\mathbb{S}^3\times\mathbb{R}$. For notational simplicity, we define the following set.

\newtheorem{Definition_3_0}[Lemma_3]{Definition}
\begin{Definition_3_0}\label{Definition_3_0}
Let $\mathcal{M}^4(\kappa,C_I)$ be the collection of all the four-dimensional ancient solutions $(M,g(t))_{t\in(-\infty,0]}$ satisfying all the following properties.
\begin{enumerate}[(a)]
  \item $g(t)$ has nonnegative curvature operator and is $\kappa$-noncollapsed for all $t\in(-\infty,0]$.
  \item The Riemann curvature tensor satisfies
   \begin{eqnarray*}
   |Rm|(g(t))\leq\frac{C_I}{|t|},
   \end{eqnarray*}
   for all $t\in(-\infty,0)$.
  \item The asymptotic shrinker of $(M,g(t))$ is the cylinder $\mathbb{S}^3\times\mathbb{R}$.
\end{enumerate}
\end{Definition_3_0}

\bigskip

We will only consider the $\kappa$-solutions in the set $\mathcal{M}^4(\kappa,C_I)$; notice that $\mathcal{M}^4(\kappa,C_I)$ is invariant under parabolic scaling. In the following lemmas, every estimate constant will be dependent on $\kappa$ and $C_I$, and this dependence will be implicit in the statement of our results.
\\

\newtheorem{Lemma_3_1}[Lemma_3]{Lemma}
\begin{Lemma_3_1}\label{Lemma_3_1}
Let $1-\bar{\mu}$ be the asymptotic reduced volume of the cylindrical shrinker $
(Cyl,g_{cyl}(t),(x_0,-1),l_{cyl})$. For any $\mu\in(0,\bar{\mu})$, there exists a $T_1=T_1(\mu)\in(-\infty,0)$, such that the following holds. Let $(M,g(t))\in\mathcal{M}^4(\kappa,C_I)$ and $x\in M$ be such that $R(x,0)=1$. Let $\mathcal{V}(t)$ be the reduced volume centered at $(x,0)$. Then $\mathcal{V}(t)<1-\mu$ whenever $t<T_1$.
\end{Lemma_3_1}

\bigskip

\noindent\textbf{Remark:} The key point of Lemma \ref{Lemma_3_1} is that $\mathbb{S}^3\times\mathbb{R}$ has the largest asymptotic reduced volume among all possible asymptotic shrinkers. Without this assumption, one may consider the following counterexample. Denote the asymptotic reduced volume of $(\mathbb{S}^2\times\mathbb{R})/\mathbb{Z}_2$ as $\xi>0$, and let $x_i\in(\mathbb{S}^2\times\mathbb{R})/\mathbb{Z}_2$ such that $x_i\rightarrow\infty$. Then we have that $\lim_{t\rightarrow-\infty} \mathcal{V}_{(x_i,0)}(t)=\xi$ for every $i$, but $\lim_{t\rightarrow-\infty}\lim_{i\rightarrow\infty}\mathcal{V}_{(x_i,0)}(t)=2\xi$, since along $x_i$ the ancient flow converges to $\mathbb{S}^2\times\mathbb{R}$. It follows that Lemma \ref{Lemma_3_1} fails for $(\mathbb{S}^2\times\mathbb{R})/\mathbb{Z}_2$.
\\

\begin{proof}[Proof of Lemma \ref{Lemma_3_1}]
Suppose the lemma is not true, then for some $\mu\in(0,\bar{\mu})$ we have a sequence of counterexamples $\{(M_i,g_i(t))\}_{i=1}^\infty\subset\mathcal{M}^4(\kappa,C_I)$ and $x_i\in M_i$, such that for every $i$ it holds that $R_i(x_i,0)=1$ but $\mathcal{V}_i(t)\geq 1-\mu$ for all $t\in[-i,0)$. Here $\mathcal{V}_i$ is the reduced volume based at $(x_i,0)$, and we let $l_i$ be the reduced distance based at $(x_i,0)$. Using Lemma \ref{Lemma_2_1} we can extract a limit Type I $\kappa$-solution $(M_\infty,g_\infty(t),(x_\infty,0),l_\infty)$ form $\{(M_i,g_i(t),(x_i,0),l_i)\}_{i=1}^\infty$. Moreover $g_\infty(t)$ is nonflat and
\begin{eqnarray*}
\mathcal{V}_\infty(t)\geq 1-\mu,
\end{eqnarray*}
for every $t\in(-\infty,0)$. It follows that $(M_\infty,g_\infty(t))$ has a nonflat asymptotic shrinker whose asymptotic reduced volume is strictly greater than $1-\bar{\mu}$. By the classification of nonnegatively curved four-dimensional shrinkers \cite{naber2010noncompact} \cite{li2016four} \cite{munteanu2017positively} and by the value table in \cite{cao2004gaussian}, we know there exists no such shrinker; this is a contradiction.
\end{proof}

\bigskip

In the rest of this section, we often work with $(M,g(t),(x,-1),l)$, where $(M,g(t))\in\mathcal{M}^4(\kappa,C_I)$, and $l$ is the reduced distance function based at some point (not necessarily the same as $x$) on time $0$ slice. We define the following notion for the convenience of statement.

\newtheorem{Definition_3_2}[Lemma_3]{Definition}
\begin{Definition_3_2}\label{Definition_3_2}
$(M,g(t),(x,-1),l)$ is said to be $\varepsilon$-close to $(Cyl,g_{cyl}(t),(y,-1),l_{cyl})$ on the interval $[a,b]$, where $\varepsilon>0$ is a small number, $y\in Cyl$, and $-1\subset(a,b)\subset[a,b]\subset(-\infty,0)$, if on the space-time products $B_{g(-1)}(x,\varepsilon^{-1})\times[a,b]$ and $B_{g_{cyl}(-1)}(y,\varepsilon^{-1})\times[a,b]$, the Ricci flows $g(t)$ and $g_{cyl}(t)$ are $\varepsilon$-close in the $C^{\lfloor\varepsilon^{-1}\rfloor}$ topology, and the functions $l$ and $l_{cyl}$ are $\varepsilon$-close in the $C^0$ topology.
\end{Definition_3_2}

\bigskip

\newtheorem{Lemma_3_3}[Lemma_3]{Lemma}
\begin{Lemma_3_3}\label{Lemma_3_3}
For every $c>0$ and $A\in(0,1)$ there exists an $\varepsilon_1=\varepsilon_1(c,A)$, such that the following holds. Let $(M,g(t))\in\mathcal{M}^4(\kappa,C_I)$. Let $x\in M$ and $l$ be the reduced distance function based at $(x,0)$. Assume that
\begin{enumerate}[(a)]
  \item $(M,g(t),(x,-1),l)$ is $\varepsilon_1$-close to $(Cyl,g_{cyl}(t),(y,-1),l_{cyl})$ on the interval $[-A^{-1},-A]$, where $y\in Cyl$, and
  \item $\displaystyle l(x,t)< \frac{3}{2}+c$ for all $t\in[-1,-A]$.
\end{enumerate}
Then $\displaystyle l(x,t)<\frac{3}{2}+c$ for all $t\in[-A^{-1},-1]$.
\end{Lemma_3_3}

\begin{proof}
Suppose by contradiction that for some $c$ and $A$ there exists a sequence of counterexamples $\{(M_i,g_i(t))\}_{i=1}^\infty\in\mathcal{M}^4(\kappa,C_I)$ and $x_i\in M_i$ with the following properties.
\begin{enumerate}[(1)]
  \item $(M_i,g_i(t),(x_i,-1),l_i)$ is $i^{-1}$ close to $(Cyl,g_{cyl}(t),(y_i,-1),l_{cyl})$ on the interval $[-A^{-1},-A]$, where $y_i\in Cyl$, and
  \item $\displaystyle l_i(x_i,t)<\frac{3}{2}+c$ for all $t\in[-1,-A]$ and for all $i$, and
  \item There exists $t_i\in[-A^{-1},-1]$ such that $\displaystyle l_i(x_i,t_i)\geq \frac{3}{2}+c$.
\end{enumerate}
First we observe that $y_i$ must be within bounded distance from $x_0$ because of (\ref{eq:l_cyl}), and hence converges to some $y_\infty\in Cyl$ by passing to a (not relabelled) subsequence. Then we can use Lemma \ref{Lemma_2_2} to extract a limit from $\{(M_i,g_i(t),(x_i,-1),l_i)\}_{i=1}^\infty$. In view of the assumption (1) above we can assert that the limit of the Ricci flows is $(Cyl,g_{cyl}(t),(y_\infty,-1))$, and the limit of $l_i$ is exactly the same as $l_{cyl}$, at least on the interval $[-A^{-1},-A]$. It follows that $\displaystyle l_{cyl}(y_\infty,t)\leq \frac{3}{2}+c$ for all $t\in[-1,-A]$ but $\displaystyle l_{cyl}(y_\infty,t_1)\geq \frac{3}{2}+c$ for some $t_1\in[-A^{-1},-1]$. On the other hand, by (\ref{eq:l_cyl}) we have that $l_{cyl}(y_\infty,t)$ is either identically $\displaystyle \frac{3}{2}$ or strictly increasing in $t$; this is a contradiction.
\end{proof}

\bigskip

\newtheorem{Lemma_3_4}[Lemma_3]{Lemma}
\begin{Lemma_3_4}\label{Lemma_3_4}
For any $\varepsilon\in(0,1)$ and $C<\infty$, there exists a $T_2=T_2(\varepsilon,C)\in(-\infty,0)$, such that the following holds. Let $(M,g(t))\in\mathcal{M}^4(\kappa,C_I)$. Let $x\in M$ such that $R(x,0)=1$ and $l$ be the reduced distance function base at $(x,0)$. For every $t_0< T$ and $y_0\in M$, if $l(y_0,t_0)\leq C$, then $(M,|t_0|^{-1}g(t|t_0|),(y_0,-1),l(t|t_0|))$ is $\varepsilon$-close to $(Cyl,g_{cyl}(t),(y,-1),l_{cyl})$ on the interval $[-\varepsilon^{-1},-\varepsilon]$. Here $y\in Cyl$ is a point not necessarily the same as $x_0$.
\end{Lemma_3_4}

\begin{proof}
Suppose by contradiction that for some $\varepsilon$ and $C$ there is a sequence of counterexamples $\{(M_i,g_i(t))\}_{i=1}^\infty\subset\mathcal{M}^4(\kappa,C_I)$, $x_i\in M_i$, $y_i\in M_i$, and $t_i\searrow-\infty$ such that the following hold.
\begin{enumerate}[(1)]
  \item $R_i(x_i,0)=1$ for ever $i$, and
  \item $l_i(y_i,t_i)\leq C$ for every $i$, where $l_i$ is reduced distance centered at $(x_i,0)$, and
  \item $(M_i,|t_i|^{-1}g(t|t_i|),(y_i,-1),l_i(t|t_i|))$ is not $\varepsilon$-close to $(Cyl,g_{cyl}(t),(y,-1),l_{cyl})$ on the interval $[-\varepsilon^{-1},-\varepsilon]$ for any $y\in Cyl$.
\end{enumerate}
By passing to a sequence, we may assume $t_i\leq \varepsilon^{-1}T_1\left(\bar{\mu}-\frac{1}{i}\right)$ for every $i$, where $T_1$ is defined in Lemma \ref{Lemma_3_1}. By Lemma \ref{Lemma_2_2} we can extract a limit from the scaled flows $\{(M_i,|t_i|^{-1}g(t|t_i|),(y_i,-1),l_i(t|t_i|))\}$, whose limit we denote as $(M_\infty,g_\infty(t),(y_\infty,-1),l_\infty)$. Let $\mathcal{V}_\infty(t)$ be the reduced volume function correspondent to $l_\infty$, we have by Lemma \ref{Lemma_2_2}(c) that $\mathcal{V}_\infty(t)\equiv 1-\bar{\mu}$ for all $t\in(-\infty,\varepsilon]$. By Lemma \ref{Lemma_2_2}(d) we have that $(M_\infty,g_\infty(t),(y_\infty,-1),l_\infty)$ is the canonical form of a Ricci shrinker and hence must be $(Cyl,g_{cyl}(t),(y,-1),l_{cyl})$ for some $y\in Cyl$, in view of the value of its asymptotic reduced volume; this is a contradiction against assumption (3) above.

\end{proof}

\bigskip

\begin{proof}[Proof of Proposition \ref{Proposition_Main}]
The proof of Proposition \ref{Proposition_Main} is implied by the following claim.

\newtheorem*{Claim}{Claim}
\begin{Claim}
For all $c>0$, there exists $\delta=\delta(c)>0$ and $T_3=T_3(c)\in(-\infty,0)$, such that the following holds. Let $(M,g(t))\in\mathcal{M}^4(\kappa,C_I)$. Let $x\in M$ and $l$ be the reduced distance based at $(x,0)$. Assume that $R(x,0)=1$ and that $(x,0)$ is the center of a $\delta$-neck, then $\displaystyle l(x,t)<\frac{3}{2}+c$ for all $t< T_3$.
\end{Claim}

\begin{proof}[Proof of the claim]
Assume the claim does not hold, then for some $c>0$, we can find a sequence of counterexamples $\{(M_i,g_i(t))\}_{i=1}^\infty\subset\mathcal{M}^4(\kappa,C_I)$ and $x_i\in M_i$, such that the following hold.
\begin{enumerate}[(1)]
  \item $R(x_i,0)=1$, and
  \item $(x_i,0)$ is the center of a $i^{-1}$-neck, and
  \item $\displaystyle l_i(x_i,\bar{t}_i)\geq \frac{3}{2}+c$, where $l_i$ is the reduced distance function based at $(x_i,0)$ and $\bar{t}_i\leq-i$.
\end{enumerate}
First we let $\displaystyle \varepsilon_1=\min\left\{\frac{1}{3},\varepsilon_1\left(c,\frac{1}{3}\right)\right\}$, where $\varepsilon_1(.,.)$ is the function defined in Lemma \ref{Lemma_3_3}. Next, we let $\displaystyle T_2=T_2\left(\varepsilon_1,\frac{3}{2}+2c\right)$, where $T_2(.,.)$ is the function defined in Lemma \ref{Lemma_3_4}. By Lemma \ref{Lemma_2_1} and by passing to a subsequence, we can assume $\{(M_i,g_i(t),(x_i,0),l_i)\}_{i=1}^\infty$ converge in the pointed smooth sense to a Type I $\kappa$-solution $(M_\infty,g_{\infty}(t),(x_\infty,0),l_\infty)$, which can be nothing but the shrinking cylinder with $R_\infty\equiv 1$ at $t=0$, since the time $0$ slice of the limit flow is exactly the cylinder, and the backward uniqueness of Kotschwar \cite{kotschwar2010backwards} implies that the limit flow is indeed the shrinking cylinder. Moreover, $l_\infty$ is the reduced distance based at $x_\infty$. Hence by Lemma \ref{Lemma_2_3}(a) and by the convergence fact, we can find a $T_0\in(-\infty,10T_2)$, such that $\displaystyle l_i(x_i,t)\leq \frac{3}{2}+\frac{c}{2}$, for all $\displaystyle t\in\left[T_0,\frac{T_0}{9}\right]$ and for all $i$ large enough. Let $\displaystyle t_i=\max\left\{t\leq T_0:l_i(x_i,t)\geq \frac{3}{2}+c\right\}$. Then we have not only that $t_i<T_0$ but also that $\displaystyle l_i(x_i,t)<\frac{3}{2}+c$ for all $t\displaystyle \in\left[\frac{t_i}{3},\frac{t_i}{9}\right]\subset(-\infty, T_2]$ and for all $i$ large enough. Moreover we have that $\displaystyle l_i(x_i,t_i)=\frac{3}{2}+c$. Now we apply Lemma \ref{Lemma_3_4} to $g_i(t)$ when $i$ is large, with $A$ taken to be $\displaystyle\frac{1}{3}$ and $t_0$ taken to be $\displaystyle\frac{t_i}{3}$. It follows that
\begin{eqnarray}\label{eq:inter}
\left(M,\frac{3}{|t_i|}g\left(t\frac{|t_i|}{3}\right),(x_i,-1),l_i\left(t\frac{|t_i|}{3}\right)\right)
\end{eqnarray}
is $\varepsilon_1$-close to $(Cyl,g_{cyl}(t),(y_i,-1),l_{cyl})$ on $\displaystyle\left[-3,-\frac{1}{3}\right]$, where $y_i\in Cyl$ and for all large $i$. Then applying Lemma \ref{Lemma_3_3} to the scaled flows (\ref{eq:inter}) with $\displaystyle A=\frac{1}{3}$, we have that $\displaystyle l_i\left(t\frac{|t_i|}{3}\right)<\frac{3}{2}+c$ for all $t\in[-3,-1]$ and for all large $i$.
In particular $\displaystyle l_i(x_i,t_i)<\frac{3}{2}+c$; this is a contradiction.
\end{proof}

\bigskip

Now we continue the proof of Proposition \ref{Proposition_Main}. Let $\delta_0$ and $\delta_1$ be the two constants in the proposition, such that $\delta_0$, $\delta_1\in(0,\min\{\delta(1),1\})$, where $\delta(1)$ is defined in the claim. We only need to let $T=\min\{T_3(1),T_2(\delta_1,4)\}$, where $T_3$ is defined in the claim and $T_2$ is defined in Lemma \ref{Lemma_3_4}.

\end{proof}

\section{Proof of the main theorem}

With the preparation in the previous sections, we are ready to prove our main theorem. This proof has essentially the same idea as our proof in \cite{zhang2017three}. On the other hand, the results in \cite{ni2009closed} and \cite{zhang2017three} are also used in the proof of Theorem \ref{Theorem_Main}. For the sake of simplicity, we continue to use the notation defined in Definition \ref{Definition_3_0}.
\\

\newtheorem{Lemma_4_1}{Lemma}[section]
\begin{Lemma_4_1} \label{Lemma_4_1}
Let $(M,g(t))\in \mathcal{M}^4(\kappa,C_I)$. Let $x_0\in M$ be a fixed point and $x_i\in M$ such that $dist_{g_i(0)}(x_i,x_0)\rightarrow\infty$. Denote $Q_i:=R(x_i,0)$ and $g_i(t):=Q_ig(tQ_i^{-1})$. Then any convergent subsequence of $\{(M_i,g_i(t),(x_i,0))\}_{i=1}^\infty$ converges to the shrinking round cylinder $\mathbb{S}^3\times \mathbb{R}$.
\end{Lemma_4_1}

\begin{proof}
By the proof of Lemma \ref{Lemma_3} and the scaling invariance of the reduced volume we have that $\lim_{t\rightarrow\infty}\mathcal{V}_i(t)=1-\bar{\mu}$ for all $i$, where $\mathcal{V}_i(t)$ is the reduced volume of $g_i(t)$ based at $(x_i,0)$ and $1-\bar{\mu}$ is the asymptotic reduced volume of the cylindrical shrinker. We denote the reduced distance function of $g_i(t)$ based at $(x_i,0)$ as $l_i$. It then follows from the monotonicity of the reduced volume that $\mathcal{V}_i(t)\geq 1-\bar{\mu}$ for every $t\in(-\infty,0)$ and for every $i$. Moreover, in view of the fact that the Type I curvature bound condition is invariant under our scaling procedure, we have that $(M,g_i(t))\in\mathcal{M}^4(\kappa,C_I)$, for every $i$. We can use Lemma \ref{Lemma_2_1} to extract a limit from $\{(M,g_i(t),(x_i,0),l_i)\}_{i=1}^\infty$, which we denote as $(M_\infty,g_\infty(t),(x_\infty,0),l_\infty)$. In particular $(M_\infty,g_\infty(t))$ is also Type I. Then we have that $\mathcal{V}_\infty(t)\geq 1-\bar{\mu}$ for all $t\in(-\infty,0)$ and that the asymptotic shrinker of $(M_\infty,g_\infty(t))$ cannot be anything but the cylinder.
\\

On the other hand, by Perelman \cite{perelman2002entropy} we know that $(M_\infty,g_\infty(t))$ splits as $(N^3\times\mathbb{R},g_N(t)+dz^2)$, where $(N^3,g_N(t))$ must be a three-dimensional Type I $\kappa$-solution, the only possibilities of which are the sphere, the round cylinder, and their quotients, see \cite{ni2009closed} and \cite{zhang2017three}. In each case $(M_\infty,g_\infty(t))$ has a shrinker structure and is self-similar. Therefore the fixed-base-point blow-down backward limit of $(M_\infty,g_\infty(t))$ is itself up to scaling. Hence $(M_\infty,g_\infty(t))$ is the shrinking cylinder.
\end{proof}

\bigskip
From this point on, we will fix a small $\varepsilon\ll\min\{\delta(\kappa,C_I),\frac{1}{100}\}$, where $\delta(\kappa,C_I)$ is defined in Proposition \ref{Proposition_Main}. The following lemma is similar to Lemma 5 in \cite{zhang2017three}.

\newtheorem{Lemma_4_2}[Lemma_4_1]{Lemma}
\begin{Lemma_4_2}\label{Lemma_4_2}
Let $(M,g(t))\in \mathcal{M}^4(\kappa,C_I)$ and $x_0\in M$ be a fixed point. Assume that $g(t)$ has strictly positive curvature operator for every $t\in(-\infty,0]$. Then for every $t\in(-\infty,0]$, there exists $x(t)\in M$, such that $(x(t),t)$ is \textbf{\emph{not}} the center of an $\varepsilon$-neck. Furthermore, $dist_{g(0)}(x_0,x(t))\rightarrow\infty$ as $t\rightarrow-\infty$.
\end{Lemma_4_2}

\begin{proof}
First we show that for every $t\in(-\infty,0]$ such $x(t)$ exists. By the Gromoll-Meyer theorem, we know that $M$ is diffeomorphic to $\mathbb{R}^4$. Suppose by contradiction that at some $t\in(-\infty,0]$ every point on $M$ is the center of an $\varepsilon$-neck. If $\varepsilon$ is taken small enough, we have that $M$ must be an $\varepsilon$-tube, and it splits as $\mathbb{S}^3\times\mathbb{R}$ since it has two ends; this is a contradiction. One may refer to Proposition A.21 in Morgan-Tian \cite{morgan2007ricci} for a detailed argument. Notice that even they work with three-dimensional geometry, in the case when every point is the center of an $\varepsilon$-neck, nothing essential dependent on the dimension is used.
\\

Next, we assume by contradiction that there exist $t_i\searrow-\infty$ such that $dist_{g(0)}(x_0,x(t_i))\leq C$, where $C$ is a constant. We show the following claim

\newtheorem*{Claim_2}{Claim}
\begin{Claim_2}
\begin{eqnarray}
dist_{g(t_i)}(x_i,x_0)&\leq& C_1+C_1\sqrt{|t_i|}, \label{eq:dist_distort}
\\
l(x_i,t_i)&\leq& C_1, \label{eq:dist_distort_1}
\end{eqnarray}
for all $i$, where $l$ is the reduced distance based at $(x_0,0)$ and $C$ depends only on $C_I$.
\end{Claim_2}

\begin{proof}[Proof of the claim]
Recall Perelman's distance distortion estimate (Lemma 8.3(b) in \cite{perelman2002entropy}): if $Ric\leq (n-1)K$ on $B_{g(t_0)}(y_0,r)\bigcup B_{g(t_0)}(y_1,r)$, where $r>0$ and $(y_0,t_0)$, $(y_1,t_0)$ are two space-time points in a Ricci flow such that $dist_{g(t_0)}(y_0,y_1)>2r$, then it holds that
\begin{eqnarray}\label{eq:dist_distort_2}
\frac{d}{dt}dist_{g(t)}(y_0,y_1)\geq-2(n-1)\left(\frac{2}{3}Kr+r^{-1}\right)
\end{eqnarray}
when $t=t_0$. Applying the Type I curvature bound and $r=|t|^{\frac{1}{2}}$ to (\ref{eq:dist_distort_2}), we have that
\begin{eqnarray}\label{eq:dist_distort_3}
\frac{d}{dt}dist_{g(t)}(x_0,x_i)\geq -C_2|t|^{-\frac{1}{2}}
\end{eqnarray}
for every $i$ and whenever $dist_{g(t)}(x_0,x_i)>2|t|^{\frac{1}{2}}$, where $C_2$ depends only on $C_I$. Integrating (\ref{eq:dist_distort_3}) from $0$ to $t_i$, (\ref{eq:dist_distort}) follows. (\ref{eq:dist_distort_1}) is an immediate consequence of Proposition 2.2 in Naber \cite{naber2010noncompact}.
\end{proof}
\bigskip

By Lemma \ref{Lemma_3} we have that $(M,|t_i|^{-1}g(t|t_i|),(x(t_i),-1))$ converges to the shrinking cylinder. This is a contradiction, since we assumed that every $(x(t_i),t_i)$ is not the center of an $\varepsilon$-neck.
\end{proof}

\bigskip

\begin{proof}[Proof of Theorem \ref{Theorem_Main}]
By Hamilton's strong maximum principle \cite{hamilton1986four}, we know that the null space of the curvature operator and the local holonomy (that is, the image of the curvature operator) must be invariant under parallel translation and independent of time. Hamilton uses maximum principle on closed manifold, one may refer to Sections 12.3 and 12.4 in \cite{chow2008ricci} for the same maximum principle on noncompact manifold. Let $\mathfrak{g}=Rm(\bigwedge^2TM)$ be the local holonomy group, we need only to consider the following $6$ cases. The details of the following argument can be found in Hamilton \cite{hamilton1986four}

\begin{enumerate}[(1)]
  \item $\mathfrak{g}=\{0\}$.\\
  In this case $g(t)$ is flat, which is not possible.

  \item $\mathfrak{g}=\mathfrak{so}(2)$.\\
  In this case $(M,g(t))$ is isometric to $\mathbb{S}^2\times\mathbb{R}^2$, whose asymptotic shrinker is not $\mathbb{S}^3\times\mathbb{R}$.

  \item $\mathfrak{g}=\mathfrak{so}(2)\times \mathfrak{so}(2)$.\\
  In this case $(M,g(t))$ is isometric to $\mathbb{S}^2\times\mathbb{S}^2$, whose asymptotic shrinker is not $\mathbb{S}^3\times\mathbb{R}$.

  \item $\mathfrak{g}=\mathfrak{so}(3)$.\\
  In this case we have that $(M,g(t))$ is isometric to $\mathbb{S}^3\times\mathbb{R}$.

  \item $\mathfrak{g}=\mathfrak{so}(3)\times \mathfrak{so}(2)$.\\
  In this case $g(t)$ has a K\"{a}hler structure for every $t$. It follows that the asymptotic shrinker of $(M,g(t))$ is a K\"{a}hler shrinker, but by Theorem 3 of \cite{ni2005ancient}, $\mathbb{S}^3\times\mathbb{R}$ is not K\"{a}hler; this case is not possible.

  \item $\mathfrak{g}=\mathfrak{so}(3)\times \mathfrak{so}(3)$.\\
  In this case we have that $(M,g(t))$ has strictly positive curvature operator at every time slice, which we will discuss in the rest of the proof.
  \end{enumerate}

Assume $g(t)$ has strictly positive curvature operator on every time slice. Let $\kappa$, $C_I$ be the constants such that $(M,g(t))\in\mathcal{M}^4(\kappa,C_I)$. Let $t_i\searrow-\infty$ and $x_i\in M$ be such that $(x_i,t_i)$ is \textbf{\textit{not}} the center of an $\varepsilon$-neck. By Lemma \ref{Lemma_4_2}, we have that $dist_{g(0)}(x_0,x_i)\rightarrow\infty$, where $x_0$ is a fixed point on $M$. We define $g_i(t):=R(x_i,0)g(tR(x_i,0)^{-1})$.

\newtheorem*{Claim_1}{Claim}
\begin{Claim_1}
\begin{eqnarray}\label{eq:tau}
\tau_i:=t_iR(x_i,0)\geq T,
\end{eqnarray}
for all large $i$, where $T=T(\kappa,\varepsilon,\varepsilon,C_I)\in(-\infty,0)$ is defined in Proposition \ref{Proposition_Main}.
\end{Claim_1}

\begin{proof}[Proof of the claim]
We consider the scaled flows $g_i(t)$. Suppose the claim is not true, by passing to a subsequence we can assume $\tau_i< T$ for every $i$. By Lemma \ref{Lemma_4_1}, we know that $(x_i,0)$ is the center of an $\varepsilon$-neck when $i$ is large. Since $R_i(x_i,0)=1$ and by Proposition \ref{Proposition_Main} we have that $(x_i,\tau_i)$ is the center of an $\varepsilon$-neck. By the definition of $g_i(t)$ and $\tau_i$, we have that $g_i(\tau_i)=R(x_i,0)g(\tau_iR(x_i,0)^{-1})=R(x_i,0)g(t_i)$, and by our assumption $(x_i,t_i)$ is not the center of an $\varepsilon$-neck in the original flow $(M,g(t))$; this is a contradiction. Notice that the $\varepsilon$-necklike property is scaling invariant.

\end{proof}

We continue the proof of the theorem. Since $\{(M,g_i(t),(x_i,0))\}_{i=1}^\infty$ converges to the shrinking cylinder, we have that for a fixed $A\in[|T|,\infty)$, every point in $B_{g_i(0)}(x_i, A)\times [-A,0]$ is $\varepsilon$-necklike when $i$ is large. In particular, $(x_i,\tau_i)$ is the center of an $\varepsilon$-neck. This yields a contradiction against the definition of $\tau_i$ and the assumption on $t_i$.

\end{proof}

\bigskip

\textbf{Acknowledgement:} The author would like to thank his doctoral advisors, Professor Bennett Chow and Professor Lei Ni, for their constant support and priceless advices.

\bibliographystyle{plain}
\bibliography{citation}

\noindent Department of Mathematics, University of California, San Diego, CA, 92093
\\ E-mail address: \verb"yoz020@ucsd.edu"

\end{document}